\documentclass[12pt, a4paper]{article}

\usepackage{a4wide}
\usepackage[english]{babel}
\usepackage[utf8x]{inputenc}

\usepackage{amsmath,amsfonts,amssymb,amsthm} 
\usepackage{graphicx} 

\usepackage{euscript} 
\usepackage{amsfonts}
\usepackage{ mathrsfs }
\usepackage{tikz}
\usepackage{amscd}
\usepackage[alphabetic]{amsrefs} 
\usepackage[all]{xy}

\newtheorem{theorem}{Theorem}
\newtheorem{lemma}{Lemma}
\newtheorem{proposition}{Proposition}

\theoremstyle{definition}
\newtheorem{definition}{Definition}

\DeclareMathOperator{\Aut}{Aut}

\DeclareMathOperator{\End}{End}
\DeclareMathOperator{\Der}{Der}
\DeclareMathOperator{\Lie}{Lie}
\DeclareMathOperator{\Inder}{Inder}

\DeclareMathOperator{\MM}{M}

\DeclareMathOperator{\Spin}{Spin}
\DeclareMathOperator{\SL}{SL}
\DeclareMathOperator{\SO}{SO}
\DeclareMathOperator{\Tr}{Tr}
\DeclareMathOperator{\ad}{ad}
\DeclareMathOperator{\pt}{pt}
\DeclareMathOperator{\vect}{vect}
\DeclareMathOperator{\triv}{triv}
\DeclareMathOperator{\TKK}{TKK}
\DeclareMathOperator{\Sl}{\mathfrak{sl}}
\DeclareMathOperator{\Sp}{\mathfrak{sp}}
\DeclareMathOperator{\So}{\mathfrak{so}}
\DeclareMathOperator{\g}{\mathfrak{g}}
\DeclareMathOperator{\ddim}{dim}
\DeclareMathOperator{\m}{\mathfrak{m}}

\newenvironment{enumerate*}%
  {\begin{enumerate}%
    \setlength{\itemsep}{1pt}%
    \setlength{\parskip}{1pt}}%
  {\end{enumerate}}

\title{Geometry of symmetric spaces of type EVI}

\author{Victor A. Petrov\footnote{St. Petersburg State University, St. Petersburg, Russia}\ \footnote{\texttt{victorapetrov@googlemail.com}}, Andrei V. Semenov\footnote{St. Petersburg State University, St. Petersburg, Russia}\ \footnote{\texttt{asemenov.spb.56@gmail.com}}}

\begin{document}

\maketitle

\begin{abstract}
We generalize Atsuyama's result on the geometry of symmetric spaces of type EVI to the case of arbitrary field of characteristic zero. We relate the possible mutual positions of two points with the classification of balanced symplectic ternary algebras (also known as Freudenthal triple systems).
    
\end{abstract}

\section{Introduction}

One of the powerful tool in studying an algebraic group is to associate to it a certain geometric structure together with an action of the group. For an isotropic group Jacques Tits constructed a simplicial complex called a \emph{spherical building} whose simplices are the parabolic subgroups defined over the base field.

For an anisotropic group this tool is no longer available, but one can use other types of closed subgroups, for example, fixed points of involutions. The resulting variety is called a \emph{symmetric space}, a very classical object of differential geometry that can be considered over an arbitrary field of characteristic not $2$. Boris Rosenfeld \cite{Ros} noted that symmetric spaces of exceptional type can be considered as an ``elliptic plane'' over some nonassociative algebra (actually, a structurable algebra in terms of Allison). In the case of real numbers Atsuyama and Vinberg independently computed a number of lines passing through two points in general position \cite{A,V}; Atsuyama also described the varieties of lines passing through two points in a special position (they also form a smaller symmetric space). We generalize this last result in the case of symmetric spaces of type EVI (that is a twisted form of $E_7/D_6+A_1$) to the case of an arbitrary field of characteristic $0$. Surprisingly, the key point is the proof is related to another invention of Jacques Tits, namely, the Tits--Kantor--K\"ocher construction.

We hope that our result may be helpful in settling the $E_7$-case of the Serre Conjecture II about algebraic groups over a field of cohomological dimension $2$ \cite{Gi}. Namely, a key question there is identification of minuscule $A_1$-subgroups in a group of type $E_7$, and such subgroups form a symmetric space of type EVI.

\section{Generalities}

\subsection{Balanced symplectic ternary algebras}

A ternary algebra $A$ equipped with a symplectic form $\langle\cdot,\cdot\rangle$ is called a \emph{balanced symplectic ternary algebra} if the product satisfies the following axioms:
\begin{enumerate}
    \item $xyz=yxz+\langle x,y\rangle z$;
    \item $xyz=xzy+\langle y,z\rangle x$;
    \item $(xyz)vw=(xvw)yz+x(yvw)z+xy(zwv).$
\end{enumerate}

For example, one can set
\begin{equation}\label{eq:triv}
xyz=\frac{1}{2}(\langle x,y\rangle z+\langle y,z\rangle x+\langle z,x\rangle y).
\end{equation}

A less trivial example is the algebra of 2 by 2 matrices of the form
$
\begin{pmatrix}
\alpha&a\\
b&\beta
\end{pmatrix}$, where $\alpha, \beta\in F$, $a,b\in J$,
for a quadratic Jordan algebra $J$, with the product defined as in \cite[\S~3,II]{FF}.

Other names for this structure include: $J$-ternary algebra in the case when $J$ is one dimensional Jordan algebra (see \cite{All}) and Freudenthal triple system, which usually considered with the symmetrized version of the product (see \cite{Fe}).


One can naturally define a notion of an ideal, and a ternary algebra is called \emph{simple} if it is nonzero and contains no nontrivial ideals.


The following proposition provides a classification of simple balanced symplectic ternary algebras.

\begin{proposition}\label{prop:class}
If $F$ is algebraically closed then any simple balanced symplectic ternary algebra with $\langle\cdot,\cdot\rangle\not\equiv 0$ is contained in the following list:
\begin{enumerate}
\item The algebra constructed from a nondegenerate symplectic form with the product as in \eqref{eq:triv};
\item The algebra constructed by a quadratic Jordan algebra associated with a nondegenerate quadratic form;
\item The algebra constructed by a quadratic Jordan algebra associated with an admissible cubic form.
\end{enumerate}
\end{proposition}
\begin{proof}
See \cite[Theorem~4.1]{FF}.
\end{proof}

\subsection{Tits--Kantor--K\"ocher construction}

We say that Lie algebra $L$ is \emph{5-graded} if $L = \bigoplus_{i \in \mathbb{Z}} L_i$ as graded Lie algebra and $L_i = 0$ for any $i \not \in \{-2,-1,0,1,2\}$. We will be interested in the special (sometimes called \emph{extra-special}) case when $\dim L_2=\dim L_{-2}=1$, say, $L_2=Fe$, $L_{-2}=Ff$, and each $L_i$ is the eigenspace of the operator $[[e,f],-]$ with the eigenvalue $i$. In this case $L_2$ and $L_{-2}$ generate the subalgebra isomorphic to $\Sl_2$.

The operation $xyz=[[[f,x],y],z]$ defines on $L_1$ the structure of a balanced symplectic ternary algebra, with the symplectic form defined by $[x,y]=\langle x,y\rangle e$.

Conversely, given a balanced symplectic ternary algebra $A$ the vector space $A\oplus A$ equipped with the operation
\begin{align*}
    [(x_1,y_1),(x_2,y_2),(x_3,y_3)]=(&x_1x_2y_3-x_1x_3y_2+x_3x_2y_1-x_2x_3y_1,\\
    &-(y_1y_2x_3-y_1y_3x_2+y_3y_2x_1-y_2y_3x_1))
\end{align*}
becomes a Lie triple system (see \cite[(1.7)]{FF}). In its own turn, a Lie triple system $M$ defines the Lie algebra $M\oplus\Inder(M)$ (which we call the embedding Lie algebra), where $\Inder(M)$ is the Lie algebra of the \emph{inner} derivations of $M$, that is of the operators $[x,y,-]$ with $x,y\in M$. In our situation the resulting Lie algebra is 5-graded with $L_1$ isomorphic to $A$; the -2-nd and 2-nd components are one-dimensional unless the symplectic form is trivial. The resulting Lie algebra is called the \emph{Tits--Kantor--K\"ocher construction} and is denoted by $\TKK(A)$.

For example, the simple balanced symplectic ternary algebra listed in Proposition~\ref{prop:class} produce the following simple Lie algebras respectively:
\begin{enumerate}
    \item The symplectic Lie algebra $\Sp_{2n}$;
    \item The orthogonal Lie algebra $\So_n$;
    \item The exceptional Lie algebras $\mathfrak{g}_2$, $\mathfrak{f}_4$, $\mathfrak{e}_6$, $\mathfrak{e}_7$, $\mathfrak{e}_8$.
\end{enumerate}

\begin{proposition}\label{prop:simple}
If $\TKK(A)$ is reductive then $A$ is simple.
\end{proposition}
\begin{proof}
Indeed, $\TKK(A)$ is perfect, since it contains an element $[e,f]$ acting as a nonzero scalar on each $L_i$, $i\ne 0$, while $L_0$ consists of inner derivations, that are commutators. Now $A$ is semisimple by \cite[Corollary~2]{FF} and simple by \cite[Lemma~3.1]{FF}.
\end{proof}

\subsection{Quaternionic gifts}\label{sec:Quat}

We will need also a twisted version of the Tits--Kantor--K\"ocher construction described in \cite{Pe}. Consider a 5-graded Lie algebra $L$ as above; then the action of $e$ and $f$ defines an isomorphism between $L_1$ and $L_{-1}$, and the action of $\Sl_2$ on $L_1\oplus L_{-1}$ agrees with the action of $\MM_2(F)$ on $L_1\oplus L_{-1}\simeq F^2\otimes L_1$ on the first component.

Assume that the symplectic form $\langle \cdot,\cdot\rangle$ is nondegenerate. Then the product defines the map
\begin{align*}
    \pi\colon\End(L_1)&\to\End(L_1)\\
    \langle \cdot,u\rangle v&\mapsto (w\mapsto uvw).
\end{align*}

The Morita equivalence produces the corresponding map
$$
\pi\colon\End_{\MM_2(F)}(F^2\otimes L_1)\to \End_{\MM_2(F)}(F^2\otimes L_1).
$$

Similarly, the symplectic form $\langle\cdot,\cdot\rangle$ on $L_1$ produces the Hermitian form
$$
\phi\colon F^2\otimes L_1\times F^2\otimes L_1\to\MM_2(F).
$$

The Lie triple system structure on $L_1\oplus L_{-1}\simeq F^2\otimes L_1$ can be reconstructed by the formula
$$
D(u,v)=\frac{1}{2}\big(\pi(\phi(\cdot,u)v-\phi(\cdot,v)u)+\phi(v,u)-\phi(u,v)\big).
$$

This description admits a Galois descent: one can take a quaternion algebra $Q$ over $F$ and consider a left $Q$-module $W$ together with a map
$$
\pi\colon\End_Q(W)\to\End_Q(W)
$$
and a Hermitian for
$\phi\colon W\times W\to Q$
that become isomorphic to the maps considered above after passing to a splitting field of $Q$. The same formula defines then a structure of a Lie triple system $W$, and we define $\TKK(W)$ to be the embedding Lie algebra. Note that $\End_Q(W)$ is a \emph{gift} (generalized Freudenthal triple system) in terms of \cite{Gar1}. In this paper, however, we will call $W$ together with $\pi$ and $\phi$ a \emph{quaternionic gift}.

\subsection{Dynkin indices}
Let $\g$ be simple Lie algebra over the base field $F$ of characteristic neither 2 nor 3, and let rank$\g = n$. Fix the set $\Phi$ of all simple roots and denote by $\Phi^+$ the set of all positive roots of $\g$. Recall that the Killing form $\kappa(x,y)$ is defined by the rule
$$\kappa (x,y) =\Tr(\ad(x) \circ \ad(y)).$$

Fix a long coroot $\theta$ and define $(-,-)$ to be an invariant scalar product on $\Phi$ normalized such that $(\theta, \theta) = 2$. Given any invariant bilinear symmetric form $B$ on $\g$ one can find the normalized Killing form $(x,y) = \frac{2B(x,y)}{B(\theta, \theta)}$.

Let $f : \g \longrightarrow \mathfrak{gl}(V)$ be a representation $V$ of $\g$. We define $(x,y)_f = \Tr(f(x)f(y))$ for all $x, y \in \g$, where $\Tr$ is just the matrix trace. This is an invariant symmetric bilinear form defined on $\g$.

Since an invariant symmetric bilinear form on $\g$ is unique up to a non-zero constant, for any $f$ of $\g$ there exist a unique $\lambda_f \in F$ such that $(x,y)_f = \lambda_f (x,y),$ where $(-,-)$ is the normalized Killing form on $\g$. The constant $\lambda_f$ is called the {\it Dynkin index} of the representation $f$ of $\g$.


In order to compute Dynkin index for an irreducible representation $f : \g \longrightarrow \mathfrak{gl}(V)$ with the highest weight $\Lambda$ one can use the following formula:
$$\lambda_f = \frac{(\Lambda, \Lambda + 2\rho) {{\ddim V}}}{\ddim \g},$$
where $\rho = {\frac{1}{2}} \sum_{\alpha \in \Phi^+} \alpha$, and the number $(\Lambda, \Lambda + 2\rho)$ is the eigenvalue of the second-order Casimir operator.

Let $\g $ and $\m $ be simple Lie algebras and let $f : \g \longrightarrow \m$ be an embedding. In this case there exists $j_f \in F$ such that $(f(x), f(y)) = j_f \cdot (x,y)$. 
We say that such $j_f$ is a {\it Dynkin index} of an embedding $f$.

A straightforward computation shows that for any representation $\psi : \m \longrightarrow \mathfrak{gl}(V)$ we have 
$$j_f = \frac{\lambda_{\psi \circ f}}{\lambda_{\psi}}.$$

Finally, if $\g$ and $\m$ are semisimple, let us consider their decomposition into a simple components:
$$\g = \bigoplus_i \g_i \text{ and } \m = \bigoplus_j \m_j.$$
So $f = \oplus f_{ij}$, where each $f_{ij}$ is a map from (simple) Lie algebra $\g_i$ to (simple) Lie algebra $\m$. Hence one can define $\lambda_{ij}$ as a Dynkin index of $f_{ij}$ and obtain the matrix $\Lambda_f = (\lambda_{ij})_{i,j}$.
\begin{definition}
We define the \emph{Dynkin multi-index} of the map $f: \g \to \m$ to be the matrix $\Lambda_f$ parametrized by the simple components of $\g$ and $\m$.
\end{definition}
It is not hard to see that the Dynkin multi-index of the composition is the matrix product of the Dynkin multi-indices of the composing maps.

\section{Description of EVI}

For a semisimple linear group $G$ denote by $\mathcal{S}(G)$ the variety of all subgroups of type $A_1$ such that the embedding of the corresponding Lie algebras is of Dynkin multi-index $(0,\ldots,1,\ldots,0)$ (one entry is $1$ and the rest are $0$'s).

\begin{proposition}\label{prop:minusc}
\begin{enumerate}
    \item $\mathcal{S}(G_1\cdot G_2)=\mathcal{S}(G_1)\coprod\mathcal{S}(G_2)$.
    \item If $G$ is simple, then $\mathcal{S}(G)$ is homogeneous in algebraic sense, that is all such $A_1$-subgroups are conjugate after passing to an algebraic closure.
\end{enumerate}
\end{proposition}
\begin{proof}
The first claim is clear, and the second follows from \cite[Theorem~3.1]{GG}.
\end{proof}

A split group of type $E_7$ contains a maximal subgroup defined by a subsystem of type $D_6+A_1$, as Borel--de Siebenthal process shows:
$$
\xymatrix{
{\bullet^0}\ar@{--}[r]&{\bullet^1}\ar@{-}[r]&{\bullet^3}\ar@{-}[r]&{\bullet^4}\ar@{-}[d]\ar@{-}[r]&{\bullet^5}\ar@{-}[r]&{\bullet^6}\ar@{-}[r]&{\bullet^7}\\
&&&{\bullet^2}
}
$$

If $G$ is any group of type $E_7$ with ${\mathcal S}(G)(F)\ne\emptyset$ (this holds in particular when $F$ has no nontrivial odd degree extensions, see \cite{Pe}), then the centralizer of a subgroup of type $A_1$ with Dynkin index $1$ is a subgroup of type $D_6$, and so ${\mathcal S}(G)$ can be identified with a symmetric space of type $EVI=E_7/D_6+A_1$.

Following \cite{V} we introduce a structure of a ``projective'' (or, better, elliptic) plane on ${\mathcal S}(G)(F)$. Both \emph{points} and \emph{lines} are the elements of ${\mathcal S}(G)(F)$, and a point is \emph{incident} to a line if and only if the respective subgroups of type $A_1$ commute.

\section{Main result}

\begin{lemma}[Pasha Zusmanovich]\label{lem:pasha}
A Lie subalgebra of an anisotropic reductive Lie algebra is also anisotropic reductive.
\end{lemma}
\begin{proof}
Follows from the fact that a reductive Lie algebra $L$ is anisotropic if and only if $\ad_x$ is a semisimple operator for all $x\in L$, the Levi--Maltsev decomposition, and \cite[Proposition~1.2]{Fa}.

Cf. the proof of \cite[Lemma~2.3 (4)]{GaGi} in the case of algebraic groups.
\end{proof}

\begin{lemma}\label{lem:D6}
Consider a minuscule subgroup of type $A_1$ in $G$ and denote its centralizer of type $D_6$ by $H$. Then $H=\Aut(W)$ for some ``quaternionic gift'' $W$ (as described in Section~\ref{sec:Quat}).
\end{lemma}
\begin{proof}
See \cite[\S~3]{Pe}.
\end{proof}

\begin{lemma}\label{lem:subgift}
Let $H_1$ and $H_2$ be the centralizers of two minuscule subgroups of type $A_1$ in $G$ with $H_1=\Aut(W)$ as in Lemma~\ref{lem:D6}. Then there is a quaternionic subgift $U$ of $W$ of (quaternionic) dimension at least $4$ corresponding to a simple balanced symplectic ternary algebra over an algebraic closure of $F$, such that
$[H_1\cap H_2,H_1\cap H_2]^\circ$ is an almost direct product of $\Aut(U)^\circ$ and another semisimple group.
\end{lemma}
\begin{proof}
Let's pass first to the algebraic closure of $F$. There is a $5$-grading on the split Lie algebra of type $E_8$ that defines a balanced symplectic ternary algebra structure on the $56$-dimensional irreducible representation of $E_7$. A subgroup of type $D_6$ inside $E_7$ stabilizes a $32$-dimensional subspace which is by inspection a subalgebra of this $56$-dimensional ternary algebra. The intersection of two such subalgebras is also a subalgebra of dimension at least $32+32-56=8$.

Since $H_1\cap H_2$ commutes with both $A_1$'s, this intersection descends to a quaternionic subspace $U$ of $W$ of quaternionic dimension at least $4$, which is a quaternionic subgift since the intersection is a subalgebra. Now $\TKK(U)$ is a Lie subalgebra of $\Lie G$, which is reductive by Lemma~\ref{lem:pasha}, and $U$ corresponds to a simple balanced symplectic ternary algebra over an algebraic closure of $F$ by Proposition~\ref{prop:simple}.

There is a map
$$
[H_1\cap H_2,H_1\cap H_2]^\circ\to\Aut(U),
$$
which is surjective at the level of Lie algebras, since $\Der(U)=\Lie \Aut(U)$ is generated by the inner derivations. The Lie algebra of the kernel is semisimple and is an ideal of the Lie algebra $[H_1\cap H_2,H_1\cap H_2]^\circ$, so it is a direct summand with the complement isomorphic to $\Der(U)$, and the claim follows.
\end{proof}

\begin{lemma}\label{lem:comm}
Let $H_1$ and $H_2$ be the centralizers of two commuting minuscule subgroups of type $A_1$ in $G$. Then $H_1\cap H_2$ is of type $D_4+A_1$.
\end{lemma}
\begin{proof}
Indeed, $H_1\cap H_2$ is the centralizer in $H_1$ of a minuscule subgroup of type $A_1$, which has type $D_4+A_1$, as the extended Dynkin diagram shows.
\end{proof}

\begin{lemma}\label{lem:8SSTA}
Every simple balanced symplectic ternary algebra of dimension at least $8$ over an algebraically closed field contains a simple subalgebra of dimension $8$.
\end{lemma}
\begin{proof}
Follows from Proposition~\ref{prop:class}. Indeed, a nondegenerate quadratic form of dimension at least $3$ contains a $3$-dimensional nondegenerate subform, and a quadratic Jordan algebra associated with an admissible cubic form contains a $3$-dimensional \'etale subalgebra.
\end{proof}

\begin{lemma}\label{lem:8in32}
Let $V$ be a $32$-dimensional balanced symplectic ternary algebra over an algebraically closed field with $\TKK(V)$ of type $E_7$. Then all $\TKK$ of $8$-dimensional simple subalgebras in $V$ are conjugate in $\TKK(V)$.
\end{lemma}
\begin{proof}
By Proposition~\ref{prop:class} $\TKK$ of an $8$-dimensional simple balanced symplectic ternary algebra $U$ has type $D_4$, and $\Inder(U)$ is of type $3A_1$.

Let us show that the Dynkin index of the embedding $D_4\to E_7$ is $1$. According to the table in \cite{Mi} the only Lie subalgebra of type $D_4$ in $E_7$ of Dynkin index more than $1$ (namely, $2$) sits inside $A_7$, and in this case the $56$-dimensional irreducible representation of $E_7$ decomposes into the sum of two copies of the adjoint representation of $D_4$. Restricting to $3A_1$ we get
$$
2\ad_1+2\ad_2+2\ad_3+4\vect_1\otimes\vect_2\otimes\vect_3+6\triv,
$$
where $\ad_i$ stands for the adjoint representation, $\vect_i$ for the vector representation, and $\triv$ for the trivial representation of the respective copy of $A_1$.

On the other hand, let us restrict this representation to $D_6$: it decomposes as the sum of two copies of the vector representation and the half-spinor representation. However, the Dynkin index of the vector representation is $2$, and the only possibility is that it decomposes as
$$
\ad_1+\ad_2+\ad_3+3\triv.
$$

Let us look at the map
$$
\SL_2\times\SL_2\times\SL_2\to\Spin_{12}.
$$
On the one hand, the image of the center becomes trivial after passing to $\SO_{12}$, and so it is contained in $\mu_2$. On the other hand, it has a representation $\vect_1\otimes\vect_2\otimes\vect_3$, so it must be at least $\mu_2\times\mu_2$, a contradiction.

Now all Lie subalgebra of type $D_4$ in $E_7$ of Dynkin index $1$ are conjugate, and $3A_1$ inside $D_4$ being the centralizer of a minuscule $A_1$ are also conjugate, and the claim follows.
\end{proof}

\begin{theorem}\label{thm:main}
Let $A$ and $B$ be two points in $\mathcal{S}(G)$ with anisotropic $G$ of type $E_7$. Then the set of lines passing through both $A$ and $B$ can be identified with the set of the rational points of a symmetric space, which has one of the following types:
\begin{enumerate}
    \item $D_6/D_4+2A_1$ (when $A=B$);
    \item $D_4/4A_1\coprod\pt$ (when $A$ and $B$ commute);
    \item $B_3/3A_1\coprod\pt$;
    \item $A_3/A_2\cdot{\mathbb G}_m\coprod\pt$;
    \item $B_2/2A_1\coprod\pt$;
    \item $\pt\coprod\pt\coprod\pt$ (in the general position);
    \item $A_5/A_3+A_1$;
    \item $C_3/C_2+A_1$.
\end{enumerate}
\end{theorem}
\begin{proof}
Denote by $H_1$ and $H_2$ the centralizers of $A$ and $B$. Then the set of lines passing through both $A$ and $B$ is the set of  subgroups of type $A_1$ in $H_1\cap H_2$ that are of Dynkin index $1$ in $G$; in particular it is contained in $\mathcal{S}(H_1\cap H_2)
=\mathcal{S}([H_1\cap H_2,H_1\cap H_2]^\circ)$.

If $A=B$, the intersection is of type $D_6$, and we are in Case~1.

If $A$ and $B$ commute, by Lemma~\ref{lem:comm} this intersection is of type $D_4+A_1$, and since the embedding $D_4+A_1\to D_6$ is of Dynkin multi-index $(1,1)$, we are in Case~2.

By Lemma~\ref{lem:subgift} $[H_1\cap H_2,H_1\cap H_2]^\circ$ is an almost direct product of $\Aut(U)^\circ$ for some quaternionic gift $U$ and a semisimple group, say, $H$. Let us pass to the algebraic closure of $F$. By Lemma~\ref{lem:8in32} the Lie algebra $\Der(U)$ of $\Aut(U)$ contains a Lie subalgebra of type $3A_1$ with the components of Dynkin index $1$ in $D_6$ (hence, in $G$). In particular, the Dynkin index of the embedding $\Der(U)\to \Lie G$ is $1$.

Further, the Lie algebra of $H$ commutes with this Lie subalgebra of type $3A_1$, and so $H$ is contained in the centralizer of $3A_1$ in $E_7$, which is of type $D_4$. On the other hand, $A$ and $B$ centralizes $H_1\cap H_2$ and hence are also contained in this subgroup of type $D_4$.

Assume that there is a minuscule $A_1$-subgroup $C$ in $G$ sitting inside $H$. Then $A$ and $B$ centralizes $C$, but the centralizer of $C$ in $D_4$ is $3A_1$, so $A$ and $B$ either coincide or commute. These cases were settled above.

Otherwise the set of lines passing through $A$ and $B$ can be identified with $\mathcal{S}(\Aut(U))$, and the 2-nd case of Proposition~\ref{prop:class}, while the 3-rd case corresponds to the Cases~7--8. The 1-st case is impossible, since the dimension of $U$ over $F$ is at least $8$, and $D_6$ doesn't contain a Lie subalgebra of type $C_n$ with $n\ge 4$.
\end{proof}

\subsection*{Acknowledgements}

We are grateful to Nikolai Vavilov and Anastasia Stavrova for the attention to our work.

\subsection*{Funding}
Theorem~1 was obtained under the support by Russian Science Foundation grant 20-41-04401. The second author was supported by Young Russian Mathematics award and by ``Native Towns'', a social investment program of PJSC ``Gazprom Neft''. Also the second author is supported in part by The Euler International Mathematical Institute, grant number 075-15-2022-289.

\newpage
\thispagestyle{empty}
Victor Alexandrovich Petrov: St. Petersburg State University, 199178 Saint Petersburg, Russia, Line 14th (Vasilyevsky Island), 29 and PDMI RAS, 191023 Saint Petersburg, Russia, nab. Fontanki, 27

\smallskip

\texttt{victorapetrov@googlemail.com}

\medskip

Andrei Vyacheslavovich Semenov: St. Petersburg State University, 199178 Saint Petersburg, Russia, Line 14th (Vasilyevsky Island), 29

\smallskip

\texttt{asemenov.spb.56@gmail.com}

\end{document}